\documentclass[11pt]{article}

\usepackage{amsmath, amsfonts, amssymb, amsthm,  graphicx, mathtools, enumerate, float}

\usepackage[blocks, affil-it]{authblk}
\usepackage{thmtools}
\usepackage{thm-restate}
\usepackage{mathrsfs}
\usepackage{todonotes}

\usepackage[round]{natbib}
\usepackage[CJKbookmarks=true,
            bookmarksnumbered=true,  
			bookmarksopen=true,
			colorlinks=true,
			citecolor=blue,
			linkcolor=blue,
			anchorcolor=red,
			urlcolor=blue]{hyperref}
\usepackage{hyperref}

\usepackage[letterpaper, left=1.2truein, right=1.2truein, top = 1.2truein, bottom = 1.2truein]{geometry}
\usepackage[ruled, vlined, lined, commentsnumbered]{algorithm2e}

\usepackage{prettyref,soul}
\usepackage[font=small,labelfont=it]{caption}

\newtheorem{lemma}{Lemma}[section]
\newtheorem{proposition}{Proposition}[section]
\newtheorem{thm}{Theorem}[section]
\newtheorem{definition}{Definition}[section]
\newtheorem{example}{Example}[section]

\newrefformat{eq}{(\ref{#1})}
\newrefformat{chap}{Chapter~\ref{#1}}
\newrefformat{sec}{Section~\ref{#1}}
\newrefformat{algo}{Algorithm~\ref{#1}}
\newrefformat{fig}{Fig.~\ref{#1}}
\newrefformat{tab}{Table~\ref{#1}}
\newrefformat{rmk}{Remark~\ref{#1}}
\newrefformat{clm}{Claim~\ref{#1}}
\newrefformat{def}{Definition~\ref{#1}}
\newrefformat{cor}{Corollary~\ref{#1}}
\newrefformat{lmm}{Lemma~\ref{#1}}
\newrefformat{lemma}{Lemma~\ref{#1}}
\newrefformat{prop}{Proposition~\ref{#1}}
\newrefformat{app}{Appendix~\ref{#1}}
\newrefformat{ex}{Example~\ref{#1}}
\newrefformat{exer}{Exercise~\ref{#1}}
\newrefformat{soln}{Solution~\ref{#1}}
\newrefformat{cond}{Condition~\ref{#1}}

\newcommand\e{\varepsilon}
\renewcommand\d{\delta}

\newcommand\R{\mathbb{R}}

\newcommand\N{\mathbb{N}}
\renewcommand\P{\mathbb{P}}
\newcommand\E{\mathbb{E}}

\newcommand{\lfdr}{\textnormal{lfdr}}

\newcommand{\iid}{\textnormal{iid}}

\newcommand{\de}{\textnormal{d}}
\newcommand{\q}{d}

\def\Real{\mathbb{R}}
\def\ascf#1#2{{#1}^{\uparrow {#2}}}

\def\dascf#1#2{{#1}^{\uparrow\!\uparrow {#2}}}

\title{Sparse-limit approximation for $t$-statistics}

\author[1]{Mic{\' o}l Tresoldi}
\author[2]{Daniel Xiang}
\author[2]{Peter McCullagh}
\affil[1]{The Dow Chemical Company} 
\affil[2]{University of Chicago}

\begin{document}
\maketitle

\begin{abstract}

In a range of genomic applications, it is of interest to quantify the evidence that the signal at site~$i$ is active given conditionally independent replicate observations summarized by the sample mean and variance $(\bar Y, s^2)$ at each site.
We study the version of the problem in which the signal distribution is sparse, and the error distribution has an unknown site-specific variance so that the null distribution of the standardized statistic is Student-$t$ rather than Gaussian.  The main contribution of this paper is a sparse-mixture approximation to the non-null density of the $t$-ratio.  This formula demonstrates the effect of low degrees of freedom on the Bayes factor, or the conditional probability that the site is active.
We illustrate some differences on a HIV dataset for gene-expression data previously analyzed by \cite{efron2012large}.


\end{abstract}

\section{Introduction}\label{sec-intro}

We consider the sparse signal plus replicated Gaussian noise model on the real line, with $m \ge 2$ conditionally independent replicate observations at each site.
By sufficiency, the observations at one site may be reduced to the sample mean and variance in the implied signal-plus-noise model
\begin{align}
\label{normal-mean-model}
\bar Y = X + \bar\e, \qquad (m-1)s^2 \sim \sigma^2 \chi_{m-1}^2.
\end{align}
The signal $X$ is assumed to be independent of the error, which is distributed as $N(0, \sigma^2/m)$ with variance inversely proportional to the sample size.  Since the model has an unknown site-specific variance parameter, it is natural first to reduce the information at each site to the standardized ratio 
$T_i = m^{1/2}\bar Y_i/s_i$, where $s_i^2$ is the sample variance.  The null distribution is Student~$t$ on $m-1$ degrees of freedom, independent for each site, and independent of~$\sigma^2$. The evidence for signal activity at site~$i$ is then based on the observed value $T_i$ in relation to the null distribution and to the values at all other sites.

This model for $(\bar{Y},s^2)$ typically arises where there is a need to summarize data with parallel structure for the sites, e.g.~in high-throughput biology settings such as microarrays (\cite{dudoit2003multiple}, \cite{efron2008microarrays}, \cite{ignatiadis2023empirical}). A common assumption in the signal detection literature is to suppose that $X$ is distributed according to a sparse `atom and slab' mixture distribution \citep{mitchell1988bayesian} $X \sim \pi_0\d_0 + (1-\pi_0) G_1$ with null probability $\pi_0 \approx 1$, and non-null component $G_1$. In this note, we work with a more general concept of sparsity \citep{mccullagh2018statistical} and derive an approximation to the marginal density of the $t$-ratio when the distribution of $X$ is sparse. To this end, up to a re-scaling of $X$, it suffices to consider a special case of (\ref{normal-mean-model}) with $\sigma=1$. 

In Section \ref{sec-sparsity-t-ratio} we review the definition of statistical sparsity and derive a mixture representation for the distribution of the $t$-ratio when the signal is sparse. In Section \ref{sec-numerical} we compute conditional probabilities using this approximation and make a comparison with the probability integral transform on a real dataset of gene expression levels from HIV patients. Section \ref{sec-discussion} concludes with a brief discussion, and Appendix \ref{sec-proofs} contains proofs of our results.

\section{Statistical sparsity and $t$-ratios}
\label{sec-sparsity-t-ratio}
\subsection{Sparse limit}
The definition of statistical sparsity proposed by \cite{mccullagh2018statistical} aims to organize the family of signal distributions by the resulting signal-plus-standard Gaussian model $X+\varepsilon$, where $\varepsilon \sim N(0,1)$ is independent of $X$. In other words, two sequences $(P_\nu)$ and $(Q_\mu)$ tend to the same sparse limit if and only if there is a 1--1 matching $\nu\mapsto\mu(\nu)$ such that the convolutions $P_\nu * N(0,1)$ and $Q_{\mu(\nu)} * N(0,1)$ are hard to distinguish as $\nu \to 0$. This criterion is formalized in terms of weak convergence of the signal distribution over a suitably large class of integrands, the constraints of which are related to particularities of the Gaussian convolutional model. 
\begin{definition}[\cite{mccullagh2018statistical}.]
A family of distributions $(P_\nu)$ indexed by $\nu > 0$ has a sparse limit with rate $\rho_\nu$ and exceedance measure $H\neq 0$ if
\begin{align}
\label{sparsity-def}
\lim_{\nu \to 0} \rho_\nu^{-1}\int_\Real w(x)\,P_\nu(\de x) =  \int_{\R} w(x)\, H(\de x) < \infty
\end{align}
for every function $w: \R \to \R$ that is bounded, continuous and $O(x^2)$ at the origin. 
\end{definition}

For every $\lambda > 0$, the pair $(\lambda \rho_\nu, \lambda^{-1} H)$ is equivalent to $(\rho_\nu, H)$ in \eqref{sparsity-def}. To eliminate this indeterminacy, a non-zero measure $H$ may be replaced with a unit measure such that
\begin{align}
\label{normalize-exceedance}
        \int_{\R} (1-e^{-x^2\!/2})\, H(\de x) = 1.
\end{align}
This particular normalization arises naturally in the signal-plus-Gaussian noise convolution.
If the family is sparse with unit exceedance measure~$H$, then the integral
\begin{equation}\label{sparsity-integral}
\int_{\Real} (1 - e^{-x^2\!/2}) \,P_\nu(\de x) = \rho_\nu + o(\rho)
\end{equation}
determines the rate parameter to first order in~$\rho\equiv \rho_\nu$. Here, $o(\rho)$ means a term $b(\rho)$ for which $\frac{b(\rho)}{\rho}\to 0$ as $\rho \to 0$.

Every distribution has its own sparsity parameter or sparsity rate as defined by~\eqref{sparsity-integral}.  By contrast, the exceedance measure is a characteristic of certain families. Thus, two sparse families $(P_\nu), (P'_\mu)$ may be put in 1--1 correspondence by their sparsity rates $\rho, \rho'$ such that $\rho(\nu) = \rho'(\mu)$.
In the sense of $w$-integrals occurring in \eqref{sparsity-def}, such matched pairs $(P_\nu, Q_\mu)$ are indistinguishable to first order in the sparse limit if the two families have the same exceedance measure.
First-order equivalence of the convolution implies an equivalence relationship on sparse families.
\begin{definition}
    Two sparse families are first-order equivalent if they have the same exceedance measure.
\end{definition}

\subsection{Scale family}\label{sec:scale}
Since this paper is concerned with the effect of unknown or estimated scale parameters, it is natural to require both the family of signal distributions and the family of error distributions to be closed under scalar multiplication.
\begin{definition}
A family of distributions $P_\nu$ indexed by $\nu > 0$ is closed under scalar multiplication if, to each pair $\nu, \sigma > 0$ there corresponds a number $\nu' > 0$ such that $X \sim P_\nu$ implies $\sigma X \sim P_{\nu'}$.  A closed family is called a scale family if the group action is transitive with a single orbit, i.e.,~to each $\nu, \nu'$ there corresponds a number $\sigma > 0$ such that $X\sim P_\nu$ implies $\sigma X \sim P_{\nu'}$.  By convention, the scale family is indexed by the scale parameter~$\sigma$.  The family has a sparse limit if \eqref{sparsity-def} is satisfied in the small-scale limit $\sigma\to 0$.
\end{definition}

A family that is closed may contain several orbits.
The Student~$t$ scale family on $\q$ degrees of freedom has one orbit for each~$\q>0$. The sparsity rate and exceedance measure for the Student $t$ scale family are stated below, shown to satisfy \eqref{sparsity-def} in Section \ref{sec-proofs}.

\begin{example}[\cite{mccullagh2018statistical}]
\label{ex:student-t}
Let $0 < \q < 2$ be given.
The Student~$t$ family on $\q$ degrees of freedom with scale parameter $\sigma>0$ has a sparse limit as $\sigma \to 0$ with rate
    \begin{align*}
        \rho = \frac{\q^{\q/2}\Gamma\left(\frac{d+1}{2} \right)}{C_\q\sqrt{\pi}\Gamma(d/2)} \cdot \sigma^\q.
    \end{align*}
The inverse-power exceedance measure
\label{lem-inverse-power-exceedance}
        \begin{align*}
        H_\q(\de x) &\coloneqq C_\q \cdot  \frac{\de x}{|x|^{\q+1}}, \qquad C_\q \coloneqq \frac{\q\, 2^{\q/2-1}}{\Gamma(1-\q/2)}
    \end{align*}
satisfies (\ref{normalize-exceedance}).  
In particular, for $\q=1$, $P_\sigma = C(\sigma)$ is the Cauchy distribution with probable error~$\sigma$;  the sparsity rate is $\rho = \sigma \sqrt{2/\pi}$, and the normalized exceedance measure is $\de x/(x^2 \sqrt{2\pi})$.
\end{example}


\begin{example}
    The zero-mean Gaussian family $P_\nu = N(0, \nu^2)$ is a scale family.  However, there does not exist a measure $H$ satisfying \eqref{sparsity-def}, so this family does not have a sparse limit.
    For the same reason, neither the scaled student~$t$ family on $d \ge 2$ degrees of freedom nor the Laplace scale family with density $\sigma^{-1} e^{-|x|/\sigma}/2$
    has a sparse limit as $\sigma\to 0$. 
\end{example}
\begin{example}
\label{ex:spike-and-slab}
    The atom-and-slab Cauchy mixture
    \[
    P_\sigma = 0.8 \delta_0 + 0.2 C(\sigma)
    \]
    is a symmetric scale family, as is the atom-free spike-and-slab mixture
    \[
    P_\sigma = 0.8 N(0, \sigma^2) + 0.2 C(\sigma).
    \]
    In both cases, \eqref{sparsity-def} is satisfied as $\nu=\sigma \to 0$ with rate $0.2\sigma\sqrt{2/\pi}$ and inverse-square exceedance density $dx/(x^2\sqrt{2\pi})$.
    Both mixtures are first-order equivalent to the sparse Cauchy model $C(\sigma)$ and to the horseshoe model \citep{carvalho2010horseshoe} with density
    $\sigma^{-1}\log(1 + \sigma^2 x^{-2})/(2\pi)$.
\end{example}

The occurrence of the Student~$t$ family indexed by $0 < d < 2$ and $\sigma > 0$ as an instance of a sparse scale family is unrelated to the principal topic of this paper, which is the distributional effect of internal normalization by the sample standard deviation.   In that case, the null distribution of $T = \bar Y \sqrt{m} / s$ is Student~$t$ on $m-1$ degrees of freedom, where the sample size $m\ge2$ is arbitrarily large and $\sigma = 1$.  The non-null distribution is derived in section~\ref{t-ratio-distn}.

The main purpose of this section is to characterize the totality of sparse scale families by the set of equivalence classes. 
Surprisingly, the Student~$t$ family in Example~\ref{ex:student-t} covers the entire range. 
In other words, every sparse scale family has an inverse-power exceedance measure with some index $0 < d < 2$.  It follows that each equivalence class has a Student~$t_d$ representative.  A precise statement is given below and a short proof is provided in Section~\ref{sec-proofs}.

\begin{thm}
\label{thm-sparse-scale}
    Every symmetric sparse scale family has a rate function $\rho(\sigma) = \sigma^\q L(\sigma^{-1})$ and an inverse-power exceedance measure $C_d \,\de x / |x|^{\q+1}$ for some $0 < \q < 2$ and slowly varying function $L:(0,\infty) \to (0,\infty)$.
\end{thm}

The claim that $\rho(\sigma) = \sigma^\q$ for scale families, which occurs
at the end of section~2.2 in \cite{mccullagh2018statistical}, is incorrect, as the following example demonstrates.
\begin{example}
    The family of unimodal symmetric distributions
    \[
    P_\sigma(\de x) = \frac{\sigma \log(1+x^2/\sigma^2)\, \de x} {2\pi\, x^2}
    \]
    is a scale family whose tails are slightly heavier than Cauchy.
    For $K \simeq 4.3552$, similar remarks apply to the log-weighted Cauchy family
    \[
    Q_\sigma(\de x) = \frac{\sigma \log(1+x^2/\sigma^2)\, \de x} {K\,(\sigma^2+x^2)},
    \]
    except that $Q_\sigma$ has modes at $\pm\sigma\sqrt{e - 1}$, with zero density at the origin.
    Nevertheless, both families satisfy the limit condition \eqref{sparsity-def} in the form
    \[
    \lim_{\sigma\to 0}\frac{P_\sigma(\de x)}{-\sigma\log\sigma} = \frac{\de x}{\pi x^2},
    \qquad
    \lim_{\sigma\to 0}\frac{Q_\sigma(\de x)}{-\sigma\log\sigma} = \frac{2\, \de x}{K x^2},
    \]
    with $L(\sigma^{-1}) \propto -\log \sigma$ in Theorem~\ref{thm-sparse-scale}.
    Accordingly, both are sparse with rate $\rho\propto-\sigma\log\sigma$, and inverse-square exceedance density.  Despite the difference in rate functions,
    they belong to the same equivalence class (Student~$t_1$) as all families in Example~\ref{ex:spike-and-slab}.
\end{example}

\subsection{Signal plus noise convolution}
Let $X$ be a random signal with distribution $P_\nu$ taken from a family of symmetric distributions having a sparse limit with rate $\rho$ and unit exceedance measure~$H$.
The observation $Y$ is generated from the signal-plus-standard Gaussian model,
\begin{align*}
Y = X + \e, \qquad \e \sim N(0,1). 
\end{align*}
Provided that $X, \varepsilon$ are independent, the marginal density of $Y$ is
\begin{align}
\nonumber
\int \phi(y - x)\, P_\nu(\de x) 
	&= \phi(y) \int e^{-x^2/2} P_\nu(\de x) + \phi(y) \int (e^{xy} - 1) e^{-x^2/2}\, P_\nu(\de x), \\
 \nonumber
	&= \phi(y) \bigl(1 - \rho + o(\rho)\bigr) +  \phi(y) \int \bigl(\cosh(xy) - 1\bigr) e^{-x^2/2}\, P_\nu(\de x), \\
 \nonumber
	&= \phi(y)(1 - \rho) +  \rho \phi(y) \int \bigl(\cosh(xy) - 1\bigr) e^{-x^2/2}\,  H(\de x) + o(\rho), \\
 \label{marginal-approx-z}
	&= (1 - \rho)\phi(y) +  \rho \phi(y) \zeta(y) + o(\rho).
\end{align}
In the second line, $\int e^{-x^2\!/2} P_\nu(dx) = 1 - \rho + o(\rho)$ follows from the definition of the sparsity rate, while symmetry of $P_\nu$ permits us to replace the exponential $e^{xy}$ with the cosh function.
Since the integrand $(\cosh(xy) - 1) e^{-x^2\!/2}$ is bounded, continuous and $O(x^2)$ at the origin, sparsity allows us replace the $P_\nu$-integral on line~2 with the $H$-integral in line~3.
The zeta function on line~4 is determined by the exceedance measure,
\begin{align*}
    \zeta(y) \coloneqq \int_{\R} (\cosh(xy)-1) e^{-x^2/2} H (\de x).
\end{align*}
Note that since $\int \phi(y) \zeta(y) \,\de y = 1$ (see, e.g.~section 3.2 of \cite{mccullagh2018statistical}), up to first order in the sparsity rate, $Y$ is distributed according to a two-component mixture, $\phi$ and $\psi = \phi \cdot \zeta$, with corresponding weights $1-\rho$ and $\rho$, where $\zeta$ is the density ratio of the non-null and null components. 

\subsection{Convolution for averages}
Consider the same set-up as in the previous section where the sample mean observation $\bar Y = X + \bar \varepsilon$ is the sum of a signal and an independent random error distributed as $N(0, 1/m)$.  For reasons given in section~\ref{sec:scale}, the family of signal distributions is assumed to be a sparse scale family with inverse-power exceedance density $H(dx) = C_\q \, dx  / |x|^{\q + 1}$.

By the same argument given in the preceding section, the marginal density of $\bar Y$ at $y$ is
\begin{align}
\nonumber
\int \phi_m(y - x)\, P_\nu(\de x) 
	&= \phi_m(y) \int e^{-mx^2/2} P_\nu(\de x) + \phi_m(y) \int (e^{mxy} - 1) e^{-mx^2\!/2}\, P_\nu(\de x), \\
 \nonumber
	&= \phi_m(y) \bigl(1 - \rho_m + o(\rho)\bigr) +  \phi_m(y) \int \bigl(\cosh(mxy) - 1\bigr) e^{-mx^2/2}\, P_\nu(\de x), \\
 \nonumber
	&= \phi_m(y)(1 - \rho_m) +  \rho_1 \phi_m(y) \int \bigl(\cosh(mxy) - 1\bigr) e^{-mx^2/2}\,  H(\de x) + o(\rho), \\
 \label{marginal-approx-zbar}
	&= (1 - \rho_m)\phi_m(y) +  \rho_m \phi_m(y) \zeta( m^{1/2}y) + o(\rho).
\end{align}
Here, $\phi_m(y) = m^{1/2} \phi(m^{1/2} y)$ is the Gaussian density with variance $m^{-1}$, while the non-Gaussian mixture fraction is
\begin{eqnarray*}
\rho_m &=& \int \bigl(1 - e^{-mx^2/2}) P_\nu(dx), \\
   &=& \rho \int \bigl(1 - e^{-mx^2/2}) C_d \, dx / |x|^{\q+1}
   = \rho\, m^{\q/2}
\end{eqnarray*}
to first order for the inverse-power measure.
Thus, the standardized statistic $m^{1/2} \bar Y$ is distributed according to the binary mixture
\begin{equation}\label{marginal-approx-z2}
(1-\rho_m)\phi(y) + \rho_m \phi(y) \zeta(y) + o(\rho)
\end{equation}
with weight $\rho_m$ on the non-Gaussian component.

The zeta function for the inverse-power exceedance measure is
expressible as a power series in $y^2$ with strictly positive coefficients
\begin{equation}\label{inverse-power-zeta}
\int_{\R}(\cosh(xy)-1) e^{-x^2/2}\, \frac{C_{\q}\,\de x} {|x|^{\q+1}}
 = -\sum_{r=1}^\infty \frac{\dascf{(-\q)} r y^{2r}} {(2r)!}
 = -\sum_{r=1}^\infty \frac{\dascf{(-\q)} r y^{2r}} {\dascf 1 r \,2^r r!}
\end{equation}
\citep{mccullagh2018statistical}.
For integer $r \ge 0$, the rising factorial function is the product $\ascf \alpha r = \alpha(\alpha+1)\cdots(\alpha+r-1)$, so that $\ascf 1 r = r!$, while $\dascf \alpha r = 2^r \ascf{(\alpha/2)} r$ is the double-step rising factorial $\dascf \alpha r = \alpha(\alpha+2)\cdots (\alpha + 2(r-1))$.  It follows that
$- \dascf{(-\q)} r$ is positive for $0 < \q < 2$.

\subsection{Mixture coefficients}
\label{sec-mixture-coefficients}
The function $h_\alpha(t) = 1 - (1 - t)^\alpha$ has a series expansion
\[
h_\alpha(t) = - \sum_{r=1}^\infty \frac{\ascf{(-\alpha)} r \,t^r} {r!}.
\]
For $0 < \alpha < 1$, the coefficients $\zeta_r = -\ascf{(-\alpha)} r / \ascf 1 r$ are strictly positive, the tail behaviour is $\zeta_r = O(\Gamma(r-\alpha) / \Gamma(r+1) = O(r^{-\alpha - 1})$ for large~$r$, and the series is convergent for $|t|\le 1$.
Since $h_\alpha(1) = 1$, the coefficients determine a probability distribution on the natural numbers.  The tail behaviour $\zeta_r = O(r^{-\alpha-1})$ implies that the moments are not finite.
The upper limit $h_1(t) = t$ is an exception implying unit mass at $x=1$. This distribution with probability generating function $h_\alpha$ arises in the following section as the coefficients in a countable mixture with inverse-power index $d = 2\alpha$. The probability distribution defined by the coefficients of $h_\alpha$ is further described in section \ref{sec-proofs}.

%

\subsection{Distribution of the $t$-ratio}\label{t-ratio-distn}
In this section, we derive the distribution of the internally standardized statistic $m^{1/2}\bar{Y}/s$ where $s$ is the sample standard deviation and $\bar{Y}$ is the sample mean among $m$ i.i.d.\ observations.
Let $Y$ denote a random variable with the distribution of $m^{1/2}\bar{Y}$, the two-component mixture \eqref{marginal-approx-z2}, and let $s^2 \sim \chi_{k}^2/k$ be a scalar mean square independent of~$Y$. The sample ratio $T \coloneqq Y /s$ also has a binary mixture representation with the same weights, where the null component is the $t$-distribution on $k=m-1$ degrees of freedom, denoted by~$f_0$, and the non-null component is determined by the zeta function relative to $f_0$. When the signal distribution has an inverse-power exceedance measure, the Student-$t$ zeta function has an explicit analytic form, stated below and proven in Section \ref{sec-proofs}.

\begin{thm}[\cite{tresoldi2021sparsity}]
\label{thm-sparse-t-marginal}
Let $0\le\rho\le 1$ be given, and suppose that $Y$ is distributed as the mixture
\begin{eqnarray*}
    Y &\sim& (1-\rho)\phi+\rho \psi, \qquad \psi(y) \coloneqq \phi(y) \zeta(y),
\end{eqnarray*} 
where $\zeta(\cdot)$ is the zeta function~\eqref{inverse-power-zeta} for the inverse-power measure $H_\q$ with index $\q \in (0,2)$, and $Y$ is independent of $s^2 \sim \chi_k^2/k$. Then the ratio $T = Y /s$ is marginally distributed according to the countably infinite mixture,
\begin{equation}
\label{marginal-approx-t}
(1-\rho) f_0(t) + \rho \sum_{r=1}^\infty \zeta_{d,r} f_r(t) 
\end{equation}
where the coefficients $\zeta_{\q,r} = -\dascf{(-\q)} r / (2^r r!)$ are non-negative and add to one, and the densities $f_r$ are given by
\begin{align}
\label{zeta_k}
    f_r(t) = \frac{t^{2r}} {( 1 + t^2/k)^{r+1/2+k/2}} \times \frac{\Gamma(1/2)} {k^{r + 1/2} \pi^{1/2} B(r+1/2, k/2)},
\end{align}
where $B(a,b)$ is the Beta function.
\end{thm}

It is convenient to write the marginal density (\ref{marginal-approx-t}) as a binary mixture
\begin{align}
\label{eq:zeta-k}
    (1-\rho)f_0(t) + \rho f_0(t) \zeta_{k}(t), \qquad \zeta_k(t) \coloneqq \sum_{r=1}^\infty \zeta_{d,r}  \frac{ t^{2r}} {\dascf 1 r} \frac{\dascf {(1+k)} r} {(k + t^2)^r}.
\end{align}
In other words, 
the density ratio $f_r(t) / f_0(t)$ for $k \geq 2$ is
\begin{equation}
\frac{f_r(t)} {f_0(t)} =  \frac{t^{2r}} {\dascf 1 r} \frac{\dascf {(1+k)} r} {(k + t^2)^r}. 
\end{equation}
These are the key elementary functions that arise in the definition of the zeta function when the null distribution is Student-$t$ rather than standard normal. Since the series expansion~\eqref{eq:zeta-k} is convergent for all~$t\in\Real$, the function is well approximated by series truncation; an implementation is provided on the website linked in section \ref{sec-discussion}.

\subsection{Probability integral transformation}
Let $g$ be the monotone transformation $\Real\to\Real$ that transforms Student-$t$ to standard normal,
i.e.,~$X \sim f_0$ implies $g(X) \sim N(0,1)$, so $g$~is the probability integral transform that sends $t$-scores to $z$-scores. More explicitly, $g(y) = \Phi^{-1}(F_{0}(y))$, where $F_0$ is the cumulative distribution function (cdf) of the $t$-distribution with $k$ degrees of freedom, and $\Phi$ is the cdf of the standard normal distribution.
\begin{proposition}
\label{prop-pit-density}
Let $\q \in (0,2)$, and suppose $T$ is distributed according to the $t$-mixture~\eqref{marginal-approx-t}. 
Then the transformed variable $Z = gT$ is distributed as a mixture with density
\begin{align}
\label{pit-marginal-density}
\phi(z)\left(1-\rho + \rho \zeta_{k}(g^{-1} z)\right),
\end{align}
where $g^{-1} z$ is the $t$-score. By contrast, if $Y$ is distributed as in Theorem \ref{thm-sparse-t-marginal}, it has density
\begin{align*}
    \phi(y) \left(1-\rho+\rho \zeta_{\infty}(y) \right),
\end{align*}
where $\zeta_\infty(y) \coloneqq \lim_{k \to \infty}\zeta_k(y)$.
\end{proposition}

Assume we observe independent copies $(Y_i,S_i^2)$ for $i=1,\dots,n$ from the model~(\ref{normal-mean-model}), with a different signal $X_i$ for each site. Classical procedures in the multiple testing literature (e.g. \cite{simes1986improved}, \cite{benjamini1995controlling}) typically take as input a list of $p$-values. This reduction of the data, obtained by the probability integral transform of the $t$-scores, subsequently ignores the degree of freedom parameter $k$ in the calculation of false discovery rates. As expression (\ref{pit-marginal-density}) indicates, the density of the non-null component depends on $m$, whereas running the $p$-values through a general purpose multiple testing procedure such as Lindsey's method (\cite{lindsey1974comparison}, \cite{lindsey1974construction}, \cite{efron2012large}) or the BH procedure effectively assumes that the convolution occurs in the space of $z$-scores. The formula obtained in Theorem \ref{thm-sparse-t-marginal} for the zeta function relative to the Student-$t$ null facilitates a comparison between these two approaches, discussed in the next section.



\section{Numerical comparison}
\label{sec-numerical}

Although there are close similarities, the probability-integral transformed $t$-mixture is not the same as
the standard Gaussian mixture
\[
(1-\rho)\phi(z) + \rho \phi(z) \zeta_\infty(z). 
\]
The zeta-function for the transformed mixture is $\zeta_{k}(g^{-1} z)$, which is close to, but not the same as $\zeta_{\infty}(z)$. The goal of this section is to demonstrate the effect of this difference on false discovery rate calculations numerically and on the HIV dataset of \cite{efron2012large}.

The two-groups model is a description of the data generating process that posits a latent variable indicating the component from which the observation arose,
\begin{align*}
    H &\sim \text{Bernoulli}(\rho) \\
    Z \mid H &\sim \begin{cases}
    \phi \hspace{1em} &\text{if }H=0 \\
    \psi &\text{if }H=1,
    \end{cases}
\end{align*}
where $\psi(z) = \phi(z)\zeta(z)$ is an alternative density satisfying $\psi(0)=0$. A quantity of interest is the local false discovery rate (\cite{efron2001empirical}), defined as the conditional probability that the latent variable $H$ is zero,
\begin{align}
\label{def-lfdr}
    \lfdr(z) = \P(H=0 \mid Z=z) = \frac{(1-\rho)\phi(z)}{(1-\rho)\phi(z)+\rho \psi(z)}.
\end{align}
The zeta function together with the sparsity rate determine the conditional odds ratio,
\begin{align*}
    \frac{1-\lfdr(z)}{\lfdr(z)} = \frac{\rho}{1-\rho} \times \zeta(z).
\end{align*}
For small degrees of freedom ($\leq 10$), the conditional odds that a signal is active can be diminished by roughly $20\%$ in the region of interest (e.g. $3 \leq |z| \leq 10$). The ratios $\zeta_{\infty}(z) / \zeta_{m}(t)$, with $z = g(t)$, are shown below for $m=10$ and three values of~$\q$:

\begin{table}[h]
\begin{center}
\caption{\label{zeta-ratio-table}Ratios $\zeta_{\infty}(z) / \zeta_{10}(t)$ for $m=10$.}
\begin{tabular}{l l l l l} 
 \hline
 $t$ & $z$ & $\q=0.5$ & $\q=1.0$ & $\q=1.5$ \\ 
 \hline
 3.0 & 2.47 &  0.84 &   0.94 &   1.05 \\
 4.0 & 3.02 &  0.79 &  0.93 &  1.10 \\
5.0 & 3.46  & 0.73 &  0.90 & 1.12 \\
7.0 & 4.12  & 0.65  & 0.85  & 1.12 \\
10.0 & 4.80 &  0.57 &  0.82 &  1.18 \\
15.0 & 5.51 &  0.52  & 0.85  & 1.37 \\\hline
\end{tabular}
\end{center}
\end{table}



It is apparent that the ratio $\zeta_{\infty}(z) / \zeta_{10}(t)$ may be appreciably less than one for heavy-tailed signals with $d<1$, or appreciably greater than one for shorter-tailed signals with $d > 1$.  Also, the ratio is not monotone as a function of the argument.


The region of most interest for comparison is typically $3 \le |z| \le 6$. Ordinarily, $k=100$ and $k=\infty$ are effectively equivalent for most statistical purposes.  In sparsity calculations, however, the dependence on the degrees of freedom is far from negligible in the region of interest, even for $k$ above $100$. 
For $\q=1$, the zeta-ratios $\zeta_\infty(z)/\zeta_{100}(z)$ at $z=3,4,5,6$ are 1.14, 1.62, 3.37 and 11.68 respectively. 

\subsection{Comparison on HIV data}


The HIV data, taken from Chapter 6 of \cite{efron2012large}, is a case-control study with gene-activity levels measured at 7680 genomic sites. For each site we compute a difference between the average gene expression levels for case and control, consisting of 4 HIV-positive and 4 HIV-negative individuals, respectively. The histogram of resulting differences is asymmetric, and although the methods we demonstrate in this section here assume a symmetric signal distribution, we still find it worthwhile to make the comparison on this dataset, as it has a small degree of freedom parameter $(k=6)$ for the pooled variance estimate.

In the analysis by \cite{efron2012large}, the goal of the study was to identify a small subset of the genes that are potentially relevant towards understanding HIV. Each difference was divided by a pooled standard error to obtain a $t$-score, and these were transformed into $z$-scores via $Z_i = \bar{\Phi}^{-1}(1-F_0(T_i))$, where $\bar{\Phi} = 1- \Phi$ is the complement of the standard normal cdf and $\bar{\Phi}^{-1}$ is its quantile function. After this pre-processing step, the $(Z_i)$ were viewed as independent samples from a two-groups model with standard Gaussian null component,
\begin{align*}
    H_i &\stackrel{\iid}{\sim} \text{Bernoulli}(\rho) \\
    Z_i \mid H_i &\sim (1-H_i) \phi + H_i \psi, \hspace{1em} \text{ independently}
\end{align*}
If $\psi$ is the non-null component defined by the inverse-power exceedance with index $\q \in (0,2)$, the maximized log likelihood relative to the null model is 48.23, which occurs at
\begin{align*}
    \hat{\rho}_z = 0.0059, \hspace{1em} \hat{\q}_z = 1.09.
\end{align*}
In formula (\ref{def-lfdr}), these values yield an estimate $\hat{\ell}_{z,i}$ of the lfdr for each site $i=1,\dots,7680$. 
Assuming the null component is instead Student-$t$, i.e.,~each $T_i$ comes from the two-groups model (\ref{marginal-approx-t}), the maximized log likelihood is 56.13, which occurs at
\begin{align*}
    \hat{\rho}_t = 0.0045, \hspace{1em} \hat{\q}_t = 0.60.
\end{align*}
Combining these estimates with the formula for the zeta function when the null is Student-$t$ gives an analogous set of local fdr estimates 
\begin{align*}
    \hat{\ell}_{t,i} \coloneqq \frac{1-\hat{\rho}_t}{1-\hat{\rho}_t + \hat{\rho}_t \zeta_k(T_i)}, \hspace{1em} i=1,\dots,7680.
\end{align*}
$\zeta_k$ can be computed quickly and accurately by truncating the infinite series representation \eqref{eq:zeta-k}. The resulting local fdr estimates for the sites with the largest (in absolute value) test statistics are displayed in Table \ref{tab:lfdr-table-d-estimated}. The BH procedure at level $\alpha=0.1$ yields 16 rejections, and the $\hat{\ell}_{t,i}$ values among these range from 0.00 to 0.42, with an average of $0.108$. Within the BH$(0.1)$ set, the $\hat{\ell}_{z,i}$ values range from $0.00$ to $0.51$, with an average of $0.144$.


\begin{table}[h]
\begin{center}
\caption{\label{tab:lfdr-table-d-estimated}Estimated local fdr values for the top 6 test statistics.}
\begin{tabular}{l l l l l} 
 \hline
 \noalign{\smallskip}
 \hspace{1em} $Z_i$ & \hspace{1em}$T_i$ & $\hat{\ell}_{z,i} \times 10^4$ & $\hat{\ell}_{t,i} \times 10^4$ & $\hat{\ell}_{z,i}/\hat{\ell}_{t,i}$ \\ \hline
 $-5.91$ & $-52.16$ &  $1.54$ &   $0.77$ &   $2.00$ \\
 $-5.74$ & $-43.89$ &  $3.93$ &  $1.95$ &  $2.01$ \\
$-5.60$ & $-38.36$  & $8.12$ &  $4.03$ & $2.02$ \\
$-5.53$ & $-35.86$  & $11.65$  & $5.78$  & $2.01$ \\
$-5.40$ & $-31.57$ &  $22.94$ &  $11.44$ &  $2.01$ \\
$-5.13$ & $-24.74$ &  $82.10$  & $41.86$  & $1.96$ \\\hline
\end{tabular}
\end{center}
\end{table}




We also compare the lfdr estimates with $d=1$ fixed, i.e., plugging $d=1$ into the formula \eqref{def-lfdr} for both $t$ and $z$ scores instead of a maximum likelihood estimate $\hat{d}_z$ or $\hat{d}_t$, which differ non-negligibly for the HIV data. 
The resulting estimates of lfdr are shown in Table \ref{tab:lfdr-table-d=1}. In other words, the factor of 2 in the last column of Table \ref{tab:lfdr-table-d-estimated} could arise in part from the difference between $\hat{d}_z$ and $\hat{d}_t$, and this effect is controlled for by fixing $d=1$ in the computation for $\hat{\ell}_z$ and $\hat{\ell}_t$.

The maximum likelihood estimates of the sparsity rates are $\hat{\rho}_z = 0.0053$ and $\hat{\rho}_t = 0.0078$, while the maximized log likelihoods are 48.13 and 52.99 relative to the null ($\rho=0$). Since the estimated power index for the $z$-scores is close to 1 ($\hat{d}_{z}=1.09$), the resulting lfdr estimates are nearly the same as the ones in Table \ref{tab:lfdr-table-d-estimated} when $d=1$ is held fixed rather than estimated.

\begin{table}[h]
\begin{center}
\caption{\label{tab:lfdr-table-d=1}Estimated local fdr values for the top 6 test statistics with $d=1$.}
\begin{tabular}{l l l l l} 
 \hline
 \noalign{\smallskip}
 \hspace{1em} $Z_i$ & \hspace{1em}$T_i$ & $\hat{\ell}_{z,i} \times 10^4$ & $\hat{\ell}_{t,i} \times 10^4$ & $\hat{\ell}_{z,i}/\hat{\ell}_{t,i}$ \\ \hline
 $-5.91$ & $-52.16$ &  $1.52$ &   $1.43$ &   $1.06$ \\
 $-5.74$ & $-43.89$ &  $3.88$ &  $3.38$ &  $1.15$ \\
$-5.60$ & $-38.36$  & $8.02$ &  $6.60$ & $1.22$ \\
$-5.53$ & $-35.86$  & $11.53$  & $9.22$  & $1.25$ \\
$-5.40$ & $-31.57$ &  $22.78$ &  $17.32$ &  $1.32$ \\
$-5.13$ & $-24.74$ &  $82.01$  & $57.30$  & $1.43$ \\\hline
\end{tabular}
\end{center}
\end{table}


The values of $\hat{\ell}_{t,i}$ among the 16 BH($\alpha=0.1$) rejections now range from $0.00$ to $0.39$ with an average value of $0.104$, whereas the values $\hat{\ell}_{z,i}$ among the rejections range from $0.00$ to $0.52$ with an average value of $0.147$.

\section{Discussion}
\label{sec-discussion}

In this paper, we derived a mixture representation of the non-null density of a $t$-ratio when the signal is sparse, and illustrated the maximum-likelihood procedure for estimating the local false discovery rate when the signal distribution has an inverse-power exceedance measure. The formula for the density is determined by the Student-$t$ zeta function and the sparsity rate $\rho$, and depends explicitly on the degrees of freedom parameter~$k$. When $k \to \infty$, the formula recovers the Bayes factor for the standard Gaussian mixture in the sparse setting. For small $k$, we have demonstrated differences in the deduced local false discovery rates, both numerically in the region of interest, and on a dataset of high-throughput gene expression levels of HIV patients. Although neither model accommodates the asymmetry that is present in the HIV data, the Student~$t$-model fits appreciably better than the transformed $z$-model in terms of the log-likelihood, and the average fitted lfdr-value within typical $\hbox{BH}_\alpha$~subsets is a reasonably close match with~$\alpha$.

The analysis presented here is agnostic to the assumption of equal variances across genes. If there were reason to believe that the variance was constant across sites, we could pool the estimates, yielding a scaled $\chi^2$ variable with essentially infinite degrees of freedom.  
Our analysis also assumes that the sparsity rate is the same at every site, and the joint distribution of $(X_i, Y_i)$ depends on $(\rho, d, \sigma_i)$. But the local false discovery formula \eqref{def-lfdr} is agnostic on the question of independence from one site to another.

\paragraph{Reproducibility.}
The R code and data used to generate Tables \ref{tab:lfdr-table-d-estimated} and \ref{tab:lfdr-table-d=1} can be found at the following Github link:

\href{https://github.com/dan-xiang/dan-xiang.github.io/tree/master/t-statistics-paper}{https://github.com/dan-xiang/dan-xiang.github.io/tree/master/t-statistics-paper}

\newpage

\appendix

\section{Proofs}
\label{sec-proofs}

\begin{proof}[Proof of Example \ref{ex:student-t}]
    Let $P_\sigma(x)=p_\sigma(x) \de x$ denote the student $t$ distribution on $d$ degrees of freedom with standard density:
    \begin{align*}
        p_1(x) &= \frac{\Gamma(\frac{d+1}{2})}{\sqrt{d\pi} \Gamma(d/2)} \left( 1 + \frac{x^2}{d}\right)^{-\frac{(d+1)}{2}}
    \end{align*}
    It suffices to show
    \begin{align*}
        \sigma^{-d} \frac{C_d \sqrt{\pi}\Gamma(d/2)}{d^{d/2}\Gamma(\frac{d+1}{2})} p_\sigma(x) \to \frac{C_d}{|x|^{d+1}}.
    \end{align*}
    The left hand side is equal to
    \begin{align*}
        \rho^{-1} p_\sigma(x) &= \sigma^{-d} \cdot \frac{C_d \sqrt{\pi}\Gamma(d/2)}{d^{d/2}\Gamma(\frac{d+1}{2})} \cdot \frac{1}{\sigma} \frac{\Gamma(\frac{d+1}{2})}{\sqrt{d\pi} \Gamma(d/2)} \left( 1 + \frac{(x/\sigma)^2}{d}\right)^{-\frac{(d+1)}{2}} \\
        &= \sigma^{-(d+1)} \cdot \frac{C_d}{d^{d/2}\sqrt{d}} \cdot \left(\sigma^{2} + \frac{x^2}{d} \right)^{-\frac{(d+1)}{2}} \sigma^{d+1}  \\
        &\to \frac{C_d}{d^{(d+1)/2}}\cdot \frac{d^{\frac{d+1}{2}}}{|x|^{d+1}} = \frac{C_d}{|x|^{d+1}}.
    \end{align*}
\end{proof}

\begin{proof}[Proof of Theorem \ref{thm-sparse-scale}]
    Let $P_\sigma(\de x) = p_\sigma(x) \de x$ denote a sparse scale family with standard density $p_1$, sparsity rate $\rho(\sigma)$, and exceedance measure $H(\de x)=h(x) \de x$. Then as $\sigma \to 0$
    \begin{align*}
        p_\sigma(x) =  \frac{p_1(x/\sigma)}{\sigma} \sim \rho(\sigma) h(x)
    \end{align*}
    for any $x \neq 0$. Taking $x=1$, the above equivalence gives
    \begin{align*}
        p_1(u) \sim h(1) \rho(u^{-1}) |u|^{-1} \text{ as } |u| \to \infty,
    \end{align*}
    which implies that for any $x > 0$, as $\sigma \to 0$
    \begin{align*}
        h(x) \sim \frac{p_1(x/\sigma)}{\sigma \rho(\sigma)} \sim \frac{h(1)\rho(\sigma/x)\cdot (\sigma/x)}{\sigma \rho(\sigma)} = h(1) \cdot\frac{ \rho(\sigma/x)}{x\rho(\sigma)}.
    \end{align*}
    Equivalently, 
    \begin{align*}
        \frac{\rho(\sigma/x)}{\rho(\sigma)} \to \frac{xh(x)}{h(1)}.
    \end{align*}
    Let $\gamma(u) \coloneqq \frac{1}{\rho(u^{-1})}$. Then the above convergence implies that for any $x>0$,
    \begin{align}
    \label{eq:gamma-regular-variation}
        \lim_{u \to \infty} \frac{\gamma(xu)}{\gamma(u)} = \frac{h(1)}{xh(x)}.
    \end{align}
    This condition implies $h(x) = h(1)x^{-d-1}$ for some $d \in \R$ (see Lemma 1 in Chapter 8 of \cite{feller1991introduction}). Therefore $\gamma(u) =u^dL(u)$ for some slowly varying function $L: (0,\infty) \to (0,\infty)$. Indeed, let $L(u) \coloneqq \gamma(u)/u^d$. Then for any $a>0$, \eqref{eq:gamma-regular-variation} implies
    \begin{align*}
        \frac{L(au)}{L(u)}=\frac{\gamma(au)u^{d}}{\gamma(u) (au)^{d}} = \frac{\gamma(au)}{\gamma(u) a^{d}} \to \frac{h(1)}{a h(a)} a^{-d} = \frac{h(1)}{a \cdot h(1) a^{-d-1}} a^{-d}=1.
    \end{align*}
    In particular, this implies $\rho(\sigma)=\sigma^{d} / L(\sigma^{-1})$. Further, we deduce $d<2$ since sparsity implies $h(x)$ 
    must be integrable against functions satisfying $w(x) = O(x^2)$ as $x \to 0$.
\end{proof}

\begin{proof}[Properties of the distribution with probability generating function $h_\alpha$ (section \ref{sec-mixture-coefficients})]
For $X$ distributed according $h_\alpha$, the reciprocal moment is the harmonic number 
\[
E(1/X) = \int_0^1 t^{-1} h_\alpha(t)\, dt =  \psi(\alpha+1) - \psi(1),
\]
where $\psi$ is the derivative of the log gamma function.  
The inverse factorial moments are
\[
E\biggl( \frac {r! } {\ascf {(X+1)} r}\biggr) = \frac \alpha {\alpha + r}.
\]
In the special case $\alpha=1/2$, the probabilities are $\zeta_r = C_{r-1}/2^{2r-1}$, where $C_r = (2r)!/(r!(r+1)!)$ is the $r$th Catalan number. Since $h_\alpha^2(t) = 2h_\alpha(t) - h_{2\alpha}(t)$ and $h_1(t) = t$, the distribution  $h_{1/2}$ has the peculiar property that the two-fold convolution is equal to the conditional distribution given $X > 1$. It is the unique distribution on the natural numbers having this property.
\end{proof}

\begin{proof}[Proof of Theorem \ref{thm-sparse-t-marginal}]

It follows from Lemma \ref{lem-zeta-inverse-power} that $\psi$ is a mixture
\begin{align*}
    \psi(y) = \sum_{r=1}^\infty \zeta_{\q,r} \cdot g_r(y) ,\hspace{1em} g_r(y) \coloneqq \frac{y^{2r}\phi(y)}{\dascf 1 r}.
\end{align*}
Note that $g_r$ is a probability density since $\int y^{2r} \phi(y) \,\de y = \dascf 1 r$ is the formula for a Gaussian moment of even degree. If $X_r \sim g_r$ and $S^2 \sim \chi^2_k$, then by Lemma \ref{lem-t-zeta-mixture-component}, the density of $T_r = X_r \sqrt{k/S^2}$ is given by
\begin{align*}
        f_r(t) &= \frac{t^{2r}} {( 1 + t^2/k)^{r+1/2+k/2}} \times \frac{\Gamma(1/2)} {k^{r + 1/2} \pi^{1/2} B(r+1/2, k/2)},
\end{align*}
so the marginal distribution of $T = Y\sqrt{k/S^2}$ can be written
\begin{align*}
    (1-\rho) f_0(t) + \rho \sum_{r=1}^\infty \zeta_{\q,r} f_r(t).
\end{align*}
\end{proof}

\begin{proof}[Proof of Proposition \ref{prop-pit-density}]
The density of $Z$ is 
\begin{align*}
\left[(1-\rho) f_0(g^{-1}z) + \rho f_0 (g^{-1}z) \zeta_{k}(g^{-1}z)\right](g^{-1})'(z)
\end{align*}
By the inverse function theorem,
\begin{align*}
(g^{-1})'(z) = \frac{1}{g'(g^{-1}z)} = \frac{\phi(z)}{f_0(g^{-1}z)},
\end{align*}
which implies the first formula. For the second formula, note that by Lemma \ref{lem-zeta-inverse-power}, the density of $Y$ can be written $(1-\rho)\phi(y)+\rho \phi(y)\zeta(y)$, where
\begin{align*}
     \zeta(y) &= \sum_{r=1}^\infty \zeta_{\q,r} \frac{y^{2r}}{\dascf 1 r}, \hspace{1em} \zeta_{\q,r} = - \frac{\dascf {(-d)} r}{2^r r!}.
\end{align*}
Now since $\lim_{k\to \infty}\frac{\dascf {(1+k)} r}{(k+y^2)^r} = 1$ for fixed $r\in \N$ and $y \in \R$, 
\begin{align*}
    \zeta_k(y) = \sum_{r=1}^\infty \zeta_{\q,r} \frac{y^{2r}}{\dascf 1 r} \frac{\dascf {(1+k)} r}{(k+y^2)^r} \to \sum_{r=1}^\infty \zeta_{\q,r} \frac{y^{2r}}{\dascf 1 r} = \zeta(y) \hspace{1em} \text{ as } k \to \infty,
\end{align*}
since the sequence $\frac{\dascf {(1+k)} r}{(k+y^2)^r}$ is eventually monotone in $k$, and the series is convergent for each fixed $k$. 

\end{proof}

\begin{lemma}
\label{lem-zeta-inverse-power}
    Suppose $Y$ distributed
    \begin{align*}
    Y &\sim (1-\rho)\phi+\rho \psi, \hspace{1em}\psi(y) \coloneqq \phi(y) \zeta(y) \\
    \zeta(y) &= \int_{\R \backslash \{0\}}(\cosh(xy)-1) e^{-x^2/2}H_\q(\de x)
\end{align*} 
for a unit inverse-power exceedance $H_\q$ with index $\q \in (0,2)$. Then 
\begin{align*}
    \zeta(y) = \sum_{r=1}^\infty \zeta_{\q,r} \frac{y^{2r}}{\dascf 1 r}, \hspace{1em} \zeta_{\q,r} \coloneqq -\frac{\dascf {(-\q)} {r}} { 2^r r!} .
\end{align*}
\end{lemma}

\begin{proof}
    By definition, the zeta function for $Y$ is
    \begin{align*}
        \zeta(y) &= \int_{\R \backslash \{0\}} (\cosh(xy)-1)e^{-x^2/2} H_\q(\de x) \\
        &= \int_{\R \backslash \{0\}}\sum_{r=1}^\infty \frac{(xy)^{2r}}{(2r)!} e^{-x^2/2} H_\q(\de x) \\
        &= \sum_{r=1}^\infty \frac{y^{2r}}{\dascf 1 r} \cdot \frac{\dascf 1 r}{(2r)!} \int_{\R \backslash \{0\}}x^{2r} e^{-x^2/2} H_\q(\de x) .
    \end{align*}
    It remains to show that
    \begin{align}
    \label{wtp-zeta-coeff}
        \frac{\dascf 1 r}{(2r)!} \int_{\R \backslash \{0\}} x^{2r}e^{-x^2/2} H_\q(\de x) = -\frac{\dascf {(-\q)} {r}} { 2^r r!} .
    \end{align}
    The left hand side can be evaluated
    \begin{align*}
    \frac{\dascf 1 r}{(2r)!} \int_{\R \backslash \{0\}} x^{2r} e^{-x^2/2} H_\q (\de x) &= \frac{\dascf 1 r}{(2r)!} \int_{\R \backslash \{0\}} x^{2r} e^{-x^2/2} \frac{\q 2^{\q/2 - 1}} {\Gamma(1-\q/2)} \frac{\de x} {|x|^{1+\q}} \\
    &= \frac{2^{\q/2} }{-\Gamma(-\q/2)2^r r!} \int_{\R \backslash \{0\}} |x|^{2r-1-\q} e^{-x^2/2} \,\de x,
    \end{align*}
    using the definition $\Gamma(z+1)=z\Gamma(z)$. Now using that $\E(|Z|^{p}) = \frac{2^{p/2}\Gamma\left( \frac{p+1}{2}\right)}{\sqrt{\pi}}$ for $p > -1$ and $Z \sim N(0,1)$, the above becomes
\begin{align*}
    &= \frac{2^{\q/2} \sqrt{2\pi}}{-\Gamma(-\q/2)2^r r!} \cdot \frac{2^{\frac{2r-1-\q}{2}}\Gamma\left( \frac{2r-1-\q+1}{2}\right)}{\sqrt{\pi}} \\
    &= \frac{\Gamma\left( r-\q/2\right)}{-\Gamma(-\q/2) r!} =\frac{\q(2-\q)\cdots (2r-2-\q)} {2 \cdot 4 \cdots (2r)}= -\frac{\dascf {(-\q)} {r}} { 2^r r!} 
\end{align*}
which are monotone decreasing at rate $O(r^{-1-\q/2})$ for large~$r$.

\end{proof}

\begin{lemma}
\label{lem-t-zeta-mixture-component}
    Suppose $X_1\sim g_r$ where $g_r(x) = \frac{x^{2r}\phi(x)}{\dascf 1 r}$ and $X_2^2 \sim \chi_k^2$ independently. Then the density of $T_1 \coloneqq X_1 \sqrt{k/X_2^2}$ is given by
    \begin{align*}
        f_r(t) = \frac{\Gamma(1/2)} {k^{r + 1/2} \pi^{1/2} B(r+1/2, k/2)} \cdot \frac{t^{2r}}{(1+t^2/k)^{r+1/2+k/2}}.
    \end{align*}
    
\end{lemma}

\begin{proof}
    Independence implies that the joint density of $X_1, X_2$ on $\Real^d\times\Real^+$ is the product
\[
x_1^{2r} e^{-x_1^2/2} x_2^{k/2-1} e^{-x_2/2} \times 
\frac 1 {(2\pi)^{1/2} \dascf 1 r} \frac 1 {2^{k/2} \Gamma(k/2)}.
\]
The transformation $(x_1, x_2) \mapsto (t_1, t_2) = (x_1 \sqrt{k/x_2}, x_2)$ has Jacobian $\sqrt{t_2/k}$, so the joint distribution of the transformed variables is
\begin{eqnarray*}
t_1^{2r} (t_2/k)^r e^{-t_2 (t_1^2 / k + 1)/2}\, t_2^{k/2-1} \sqrt{t_2/k} \times \frac{1}{(2\pi)^{1/2}\dascf 1 r 2^{k/2}\Gamma(k/2)} \\
= \frac{1} {k^{r + 1/2}(2\pi)^{1/2}\dascf 1 r 2^{k/2}\Gamma(k/2)} t_1^{2r} \cdot t_2^{r+1/2+k/2 - 1} e^{-t_2 (t_1^2 / k + 1)/2} .
\end{eqnarray*}
We recognize the pdf of the Gamma$\left(r+1/2+k/2,(1+t_1^2/k)/2\right)$ distribution in the above expression, and integrate over $t_2\in\Real^+$ to obtain the marginal density of $T_1$ at $t_1\in\Real$:
\begin{eqnarray*}
f_r(t_1) &=&\frac{t_1^{2r}} {( 1 + t_1^2/k)^{r+1/2+k/2}} \times
	\frac{\Gamma(r+1/2+k/2)\, 2^{r+1/2+k/2}}  {k^{r + 1/2} (2\pi)^{1/2} \dascf 1 r \, 2^{k/2} \Gamma(k/2)} \\
&=& \frac{t_1^{2r}} {( 1 + t_1^2/k)^{r+1/2+k/2}} \times \frac{\Gamma(1/2)} {k^{r + 1/2} \pi^{1/2} B(r+1/2, k/2)},
\end{eqnarray*}
where $B(\cdot,\cdot)$ is the beta function. Note that $f_0(\cdot)$ is the Student-$t$ density on $k$~degrees of freedom on $\Real$.

\end{proof}

\bibliographystyle{dcu}
\bibliography{reference}

\end{document}